\documentclass{amsart}

\usepackage{amssymb, url}
\usepackage{amsmath}
\usepackage[normalem]{ulem}
\usepackage{color, comment}
\newtheorem{theorem}{Theorem}
\newtheorem{lemma}[theorem]{Lemma}

\newtheorem{remark}[theorem]{Remark}
\newtheorem{construction}[theorem]{Construction}

\begin{document}

\title[Geometries with trialities coming from maps with Wilson trialities]{Incidence geometries with trialities coming from maps with Wilson trialities}

%% (obligatory)
%% The name and the addresses of the first author ...
%% \author{F. Second_name}
%% \address{Street XY, Town, State}
%% current address, e-mail and url are optional
%% \curraddr{...}
%% \email{...@...}
%% \urladdr{...}

\author{Dimitri Leemans}
\address{Universit\'e Libre de Bruxelles, D\'epartement de Math\'ematique, C.P.216 - Alg\`ebre et Combinatoire, Boulevard du Triomphe, 1050 Brussels, Belgium and Department of Mathematics and Mathematical Statistics, Ume\aa\; University,
901 87 Ume\aa, Sweden
}
\curraddr{}
\email{Leemans.Dimitri@ulb.be}
\urladdr{}

%% (optional)
%% If there are more authors, then second author, third author contains
%% the same items

\author{Klara Stokes}
\address{Department of Mathematics and Mathematical Statistics, Ume\aa\; University,
901 87 Ume\aa, Sweden}
\email{klara.stokes@umu.se}

%% (optional) If any thanks for the financial supports, grants, ...
\thanks{This research was accomplished when Dimitri Leemans was a Guest Visiting Professor at the University of Ume\aa, thanks to the Knut and Alice Wallenberg Foundation.}

\thanks{}

%% (optional) Keywords

\keywords{}

%% (obligatory)
%% AMS Classification 2000
%% The Primary classification is obligatory,
%% the Secondary classification is optional.
%% \subjclass{primary}{secondary}
%% f.e. \subjclass{35R35, 49M15, 49N50}{} or \subjclass{35R35, 49M15}{49N50}

\subjclass{51A10,51E24,20C33}{}%{}

%% (optional) Abstract
\begin{abstract}
Triality is a classical notion in geometry that arose in the context of the Lie groups of type $D_4$. Another notion of triality, Wilson triality, appears in the context of reflexible maps. We build a bridge between these two notions, showing how to construct an incidence geometry with a triality from a map that admits a Wilson triality. We also extend a result by
Jones and Poulton, showing that for every prime power $q$, the group ${\rm L}_2(q^3)$ has maps that admit Wilson trialities but no dualities.
\end{abstract}

\maketitle

%%%%% private macros, f.e. the different environments
% the environment proof and proof* is defined automaticaly

%%%% the main article
\section{Introduction}

In seven dimensional projective space, the quadric $Q$ defined by the equation $x_0x_7+x_1x_6+x_2x_5+x_3x_4=0$ features a property called triality. The triality resembles the duality of projective space in that it permutes geometric objects of distinct types while preserving the geometric incidence structure. Duality of projective space ${\rm PG}(n,\mathbb{F})$ turns the Hasse diagram of the projective subspaces upside-down; in two dimensions it sends points to lines and lines to points; in three dimensions it sends points to planes and planes to points, while it is fixing the lines.

For an incidence geometry to have a non-trivial triality it must have at least three types of geometric objects. The triality of the quadric $Q$ acts on an incidence geometry of four types: the points, the lines and the two types of maximal projective subspaces contained within the quadric.  The maximal projective subspaces are three-spaces, and they divide in two equivalence classes $A$ and $B$ so that two three-spaces are of the same type if and only if they intersect in a projective subspace of odd dimension. The triality sends the points of $Q$ to subspaces of type $A$, subspaces of type $A$ to subspaces of type $B$, and subspaces of type $B$ to points on $Q$, while the lines are sent to lines.
%Klara: The above was confusing because it was unclear what happened to the lines.

The first known description of this triality is due to Study (see~\cite[Page 435]{Porteous}, see also~\cite{Study1} and~\cite{Study2}\footnote{ See \url{http://neo-classical-physics.info/uploads/3/4/3/6/34363841/study-analytical_kinematics.pdf} 
 for an english translation of~\cite{Study2}.}). In the beginning of the 20th century he used the quadric $Q$ to  parametrize the motions of three-space.
Cartan was the first to use the term triality in~\cite{Cartan}, defining the phenomenon in the  context of Lie groups. 
  Later, Freudenthal further explored triality in the context of Lie algebras \cite{Freudenthal}.
  
  In 1959, Tits classified in~\cite{Tits1959} the trialities with absolute points of the quadric $Q$.
He observed that the set of absolute points and the incident lines form a rank two incidence geometry with the property that the girth of the incidence graph is two times the diameter. This discovery motivated his definition of generalized polygon - the content of the appendix of~\cite{Tits1959}.
In 1978, Doro introduced in~\cite{Doro1978} the notion of groups with triality, inspired by observations from Glauberman~\cite{Glauberman1968}.
%  G. Glauberman, On loops of odd order II, J. Algebra 8 (1968) 393–414.                                                                                                                                  
%[6] S. Doro, Simple Moufang loops, Math. Proc. Camb. Phil. Soc. 83 (1978) 377–392.                                                                                                                       
In a recent book~\cite{Hall2019}, Hall showed that Moufang loops and groups with triality are "essentially the same thing". The Moufang loop of the classical triality $D_4$ is (essentially) the multiplicative loop of the unit octonions.

While duality occurs more or less frequently in a variety of contexts,
the brief historical review we just made shows that triality was primarily associated with a quadric in seven dimensions, until Doro introduced the abstract notion of groups with triality. According to Doro, a group $G=\langle T\rangle$, generated by a single conjugacy class of involutions $T$, is a group with a triality if there is a surjective group homomorphism  $\pi:G\rightarrow S_3$ and for all $t,r\in T$ we have that if $\pi(t)\neq \pi(r)$ then $|\pi(r)\pi(t)|=3$.
%=\{\sigma, \rho\}$ satisfying                                                                                                                                                                            
%\rho$x^{-1}x^{\sigma}\left(x^{-1}x^{\sigma}\right)^{\rho}\left(x^{-1}x^{\sigma}\right)^{\rho^2}=1$ {\color{red} for all $x\in G$}.                                                                       
This generalizes the triality of the group ${\rm PSO}(8)$ (visible in its universal cover ${\rm Spin}(8)$, see for instance~\cite[Chapter 8]{Conway} for more details on ${\rm SO}(8)$ and ${\rm Spin}(8)$), which is the group naturally acting on the quadric $Q$, as well as the Lie group of Cartan's triality. 
It led, among other things, to the classification of finite Moufang loops by Liebeck~\cite{Liebeck1987}.

Doro's definition allows to work with triality in an abstract manner, dealing only with the group.
However, the definition requires $G$ to be generated by a single conjugacy class of involutions. Also, the group of type-permuting correlations is required to contain involutions, in particular it has
 to be the symmetric group $S_3$ and cannot be, for example, the cyclic group $C_3$.

In a previous paper~\cite{LeemansStokes} we focused on a different kind of generalization of triality; constructing flag-transitive incidence geometries with type-permuting correlation group containing an element of order three.
We also provided a coset geometry construction and a description of the geometry in terms of a reduced incidence graph: a labeled quotient graph of the incidence graph under the triality.

Within the theory of maps there is also a notion of triality given by Wilson~\cite{Wilson1979}; a triality is an operator of order three on the space of maps defined as the composition of the duality operator and the Petrie duality operator. A map has a triality if it is isomorphic to the image of this operator of order three.  

Maps featuring trialities but no dualities seem to be rarely occurring. The first infinite family was provided by Jones and Poulton in 2010; a family of reflexible maps for the groups ${\rm L}_2(2^{3n})$ \cite{JP2010}. 

In this paper, we extend Jones and Poulton's result from 2010 by showing that there is a family of reflexible maps also exist for the groups ${\rm L}_2(q^3)$ with $q$ odd.
We then prove that any map with a Wilson triality gives rise to an incidence geometry with a triality, thereby providing plenty of examples of incidence geometries with trialities, with and without dualities, coming from the literature of maps.

\section{Definitions and notation}\label{definitions}
\subsection{Maps}
For a comprehensive introduction to maps, we refer to~\cite{BS1985}.
A {\em map} is a 2-cell embedding of a graph into a closed surface without boundary.
A map $\mathcal M$ has a vertex set $V := V(\mathcal M)$, an edge set $E := E(\mathcal M)$ and a set of faces $F := F(\mathcal M)$. The set $V\cup E\cup F$ is the set of {\em elements} of $\mathcal M$.
A triple $T := \{v,e,f\}$ with $v\in V$, $e\in E$ and $f\in F$ is called a {\em flag} if each element of $T$ is incident to the other two elements of $T$.
The faces of $\mathcal M$ are the simply connected components of the surface obtained by removing the embedded graph from the surface.

An {\em automorphism} of a map is a permutation of its elements preserving the sets $V$, $E$ and $F$ and incidence between the elements.
The set of all automorphisms of a map $\mathcal M$ together with composition forms a group denoted by ${\rm Aut}(\mathcal M)$.
A map ${\mathcal M}$ is {\em reflexible} if its automorphism group has a unique orbit on its set of flags.
%{\color{red}\sout{Since we want to deal with Petrie duals, we will see in Section~\ref{wilson} that we can restrict ourselves to reflexible maps.}}
%\sout{{\color{blue}we explain further down in this section why we restrict to reflexible maps}}

If a map $\mathcal M$ is reflexible, then we can pick an arbitrary flag $F := \{v,e,f\}$ of $\mathcal M$ and the group ${\rm Aut}(\mathcal M)$ is generated by three special involutions, $\rho_0$, $\rho_1$ and $\rho_2$ such that $\rho_0$ and $\rho_2$ commute, $\rho_0$ exchanges $v$ with one of its neighbours along the edge $e$ on the border of $f$, $\rho_1$ exchanges the edge $e$ with one of its neighbours containing $v$ along the border $f$ and $\rho_2$ exchanges $f$ with a face containing $v$ and $e$.
Let $p$ be the order of  $\rho_0\rho_1$, $q$ the order of $\rho_1\rho_2$ and $r$ the order of $\rho_0\rho_2\rho_1$. We then say that $\mathcal M$ is of {\rm type} $\{p,q\}_r$.

One can also define a map to be a transitive permutation representation of the abstract group $\Gamma=\langle r_0,r_1,r_2|r_0^2=r_1^2=r_2^2=(r_0r_2)^2=1\rangle$ on the set of flags \cite{JoTh1983}. 

The monodromy group of the map is the permutation group on the flags induced by $\Gamma$. The map is reflexible  if and only if its monodromy group is isomorphic to its automorphism group \cite{JoJoWo2008}.

Wilson introduced in~\cite{Wilson1979} a group $\Sigma\cong S_3$ of six operations on reflexible maps, generated by the duality operation $D$ that transposes vertices and faces, and the Petrie duality $P$ that transposes faces and Petrie polygons. 
He pointed out that the Petrie duality is not a closed operator on maps that are not reflexible. Recently, the action of $\Sigma$ on non-reflexive maps has also received attention, see for example \cite{AbEl2022}.

Jones and Thornton showed that the group $\Sigma$ is the outer automorphism group ${\rm Out}(\Gamma)={\rm Aut}(\Gamma)/{\rm Inn}(\Gamma)\cong S_3$ \cite{JoTh1983}. 
%Any reflexible map $\mathcal M$ is obtained as a quotient of $\Gamma$ by a normal subgroup $N$.
%Klara: this sentence didn't have a dot - did you finish it?

Both the operations $D$ and $P$ have order $2$. The other elements of $\Sigma$ are the identity operation, $DPD = PDP$ (a duality transposing vertices and Petrie polygons) and two mutually inverse operations, $DP$ and $PD$, of order $3$, that he called {\em trialities}, each permuting vertices, faces and Petrie polygons in a cycle of length 3. The Wilson operators act on the monodromy group of a map in a hexad manner, see Figure~\ref{classes}. 

\begin{figure}
\begin{center}
\begin{picture}(135,160)
\put(50,150){$(\rho_0,\rho_1,\rho_2)$}
\put(0,100){$(\rho_2,\rho_1,\rho_0)$}
\put(75,140){\line(-2,-1){60}}
\put(25,130){$D$}
\put(75,140){\line(2,-1){60}}
\put(125,130){$P$}
\put(100,100){$(\rho_0\rho_2,\rho_1,\rho_2)$}
\put(15,90){\line(0,-1){30}}
\put(0,70){$P$}
\put(135,90){\line(0,-1){30}}
\put(140,70){$D$}
\put(0,50){$(\rho_0\rho_2,\rho_1,\rho_0)$}
\put(100,50){$(\rho_2,\rho_1,\rho_0\rho_2)$}
\put(15,40){\line(2,-1){60}}
\put(25,20){$D$}
\put(135,40){\line(-2,-1){60}}
\put(125,20){$P$}
\put(50,00){$(\rho_0,\rho_1,\rho_0\rho_2)$}
\end{picture}
\caption{Wilson's operations on the monodromy group of a map.}\label{classes}
\end{center}
\end{figure}
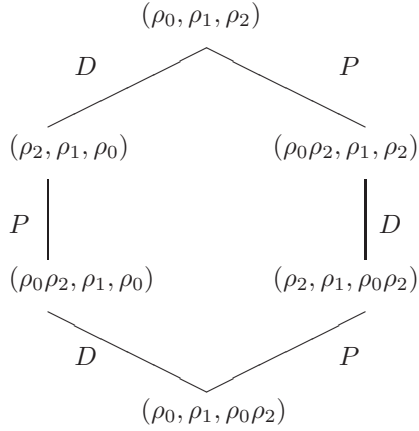

The images of a map under these six operations give maps that may be isomorphic to each other. Wilson divides regular maps into four classes. Class I is the one where the six maps obtained are pairwise non-isomorphic.
Class II is the one where the six maps split in three pairs of isomorphic maps. Class III is the one where the six maps split in two pairs of triples of pairwise isomorphic maps. And finally class IV contains the maps for which all the six images are pairwise isomorphic. 

In other words, Class I consists of the maps with no duality, nor Petrie duality, nor triality, Class II consists of the maps with dualities or Petrie dualities or a duality that exchanges the vertices with the Petrie polygons (obtained by $DPD = PDP$),
%klara: think this was wrong in the text, it said "exchanges the faces with the Petrie polygons"
but no trialities, Class III consists of the maps with trialities but no dualities, nor Petrie dualities, and Class IV consists of the maps with dualities, Petrie dualities and trialities.

Jones and Thornton showed that every finite map has a finite reflexible cover of Class IV \cite{JoTh1983}. Richter, \v{S}ir\'a\v{n} and Wang proved that there is a Class IV map of every even valency~\cite{RiSiWa2012}. 
This was later extended to odd valency by Fraser, Jeans and \v{S}ir\'a\v{n} \cite{FJS2019}. 
The kaleidoscopic maps due to Archdeacon, Conder and \v{S}ir\'a\v{n} are also, by definition, of Class IV \cite{ACS2014}.  

The first sporadic example of a map of Class III was constructed by Wilson in 1979 \cite{Wilson1979}. 
It seemed to him at the time that maps of Class III are rare. As reported by Jones and Poulton, Conder found more sporadic examples of Class III maps in a computer search in 2006, but over 30 years passed before Jones and Poulton~\cite{JP2010} produced an infinite family of reflexible maps of Class III with automorphism group ${\rm L_2}(2^{3n})$, for $n$ a positive integer, thereby extending Wilson's  first example. In the same article Jones and Poulton also constructed maps of Class III as covers of other maps of Class III as well as parallel products.
In a recent paper Abrams and Ellis-Monaghan also constructed non-reflexible maps of Class III \cite{AbEl2022}.

%In a recent paper, Abrams and Ellis-Monaghan construct all self-trial maps on up to seven edges and an infinite family of maps with trialities but no dualities that do not arise as covers or parallel connections of regular (reflexible?) maps, thereby answering a question by Jones and Poulton \cite{AbEl2019}. The maps they construct are not reflexible.

Examples of maps of small genus of all four classes can be found in Conder's list of regular maps \cite{Conder2009}. 

\subsection{Coset geometries and incidence geometries}
For a comprehensive introduction to coset geometries and incidence geometries, we refer to~\cite{Bue95}.
A {\em pregeometry} is a quadruple $\Gamma=(X,*,t,I)$ where $X$ is a set of {\em elements}, $I$ is a finite set of {\em types}, $t$ is a surjective {\em type function} $t:X\twoheadrightarrow I$ and $*$ is an {\em incidence relation}, that is a binary relation on the set $X$, such that elements of the same type are not related. 
The cardinality $|I|$ is called the {\em rank} of $\Gamma$. 
The relation $*$ is called the {\em incidence relation} of $\Gamma$. 
As any symmetric binary relation, it can be represented using a symmetric matrix or an undirected graph. The graph representing the incidence relation is called the {\em incidence graph} of the incidence geometry. It is an $n$-partite graph (with $n := |I|$), with the elements of each type in each partition. 
 
A {\em flag} of a pregeometry is a set of pairwise incident elements. The {\em rank} of a flag is the number of types in the flag. Since all elements in a flag have different type, this coincides with the number of elements in the flag. The type of a flag $F$ is the set $t(F)$.
A {\em chamber} is a flag of type $I$.
A pregeometry  is an {\em (incidence) geometry} if every flag is contained in a chamber. 
A geometry is {\em firm} if every non-maximal flag is contained in at least two chambers.

The {\em residue} of a flag $F$ in a geometry $\Gamma$ is the set of elements in $\Gamma$ that are not in $F$ but that are incident with all elements of $F$. The residue of a flag in a geometry is a geometry. 
The rank of a residue of a flag is the number of types in the residue.  
Since $\Gamma$ is a geometry, every flag is contained in a chamber, therefore the rank of the residue of a flag of rank $k$ in a geometry of rank $n$ is $n-k$. 

A {\em correlation} of an incidence geometry $\Gamma(X,*,t,I)$ is a permutation $\alpha$ of $X$ such that for all $x,y\in X$, $t(x) = t(y) \iff t(\alpha(x)) = t(\alpha(y))$ and $x*y \iff \alpha(x) * \alpha(y)$.
The set of all correlations of a geometry $\Gamma$, together with composition, forms a group called the {\em correlation group} of $\Gamma$ and denoted by ${\rm Cor}(\Gamma)$.
It contains a subgroup consisting of all the elements that induce the identity on the set of types (that is, correlations $\alpha$ such that $t(x) = t(\alpha(x))$ for all $x\in X$). This subgroup is called the group of {\em automorphisms} of $\Gamma$ and it is a normal subgroup of ${\rm Cor}(\Gamma)$ that we denote by ${\rm Aut}(\Gamma)$.

If $\Gamma$ is an incidence geometry of rank $n \geq r$, then an $r$-{\em ality} of $\Gamma$ is a  correlation of order $r$ of $\Gamma$, permuting $r$ types cyclically. 
When $r=2$ we talk about {\em duality} (when emphasis is on the duality being of order two also known as a {\em polarity}), and when $r=3$ we talk about {\em triality}.

Following Tits~\cite{Tits1957}, a coset pregeometry is a pregeometry $\Gamma$ constructed from a group $G$ and a set of subgroups $\{H_1,\dots,H_n\}$ of $G$ (called maximal parabolic subgroups), so that 
\begin{itemize}
\item the type set is $I=\{1,\dots,n\}$, 
\item the elements are the left cosets of $H_i$, 
\item the type function maps every coset $gH_i$ to $i$ and
\item two elements are incident if their intersection is not empty. 
\end{itemize}
A coset pregeometry that is a geometry is called a {\em coset geometry}.

\section{All groups ${\rm L}(2,q^3)$ admit at least one class III map}
In order to prove the claim of the section title, we will use the subgroup structure of ${\rm L}_2(q)$.

In even characteristic, Jones and Poulton~\cite[Theorem 2.3]{JP2010} showed that every group ${\rm L}_2(q^3)$ is the automorphism group of at least one class III map. Their proof is constructive. They provide generating elements of ${\rm L}_2(2^{3n})$ that give a reflexible map.

We now deal with the odd characteristic. But first, we recall what is the subgroup structure of ${\rm L}_2(q)$.

\subsection{The subgroup structure of ${\rm L}_2(q)$}
\label{pslsub}

We will require properties of the subgroup structure of ${\rm L}_2(q)$.
The subgroup structure of ${\rm L}_2(q)$ was first obtained in papers by Moore and Wiman. It may be found in Dickson~\cite{Dic58}.
The theorem reproduced here was obtained by Patricia Vanden Cruyce in her PhD Thesis~\cite{Van85}.

\begin{theorem}\label{sub}

The group ${\rm L}_2(q)$ of order $\frac{q  (q^2 - 1)}{(2,q-1)}$, where $q=p^r$ with $p$ a
prime and $r$ a positive integer, contains only:

\begin{enumerate}
\item elementary abelian subgroups of order $q$, denoted by $E_q$.

\item cyclic subgroups of order $d$,  denoted by $d$, for all divisors $d$ of $\frac{(q\pm 1)}{(2,q-1)}$.

\item $\frac{q(q^2 -1)}{2d (2,q-1)}$ dihedral groups of order $2d$, denoted by $D_{2d}$,
for all divisors $d$ of $\frac{(q\pm 1)}{(2,q-1)}$ with $d>2$. The number of conjugacy
classes of these subgroups is $1$ if $\frac{(q\pm 1)}{d (2,q-1)}$ is odd, and $2$ if
it is even.

\item For $q$ odd, $\frac{q(q^2 -1)}{12 (2,q-1)}$ dihedral groups of order 4 (Klein 4-groups), denoted by $2^2$. The number of conjugacy classes of these groups is $1$ if  $q \equiv \pm 3 (
8)$ and $2$ if $q \equiv \pm 1 ( 8)$. For $q$ even, the groups $2^2$ are listed under
family~5.

\item elementary abelian subgroups of order 
$p^s$, denoted by $E_{p^s}$, for all natural number $s$ such that $1\leq s \leq r-1$.

\item subgroups $E_{p^s}\!:\!h$,
each a semidirect product of an elementary abelian group $E_{p^s}$ and a cyclic 
group of order $h$, for all natural numbers $s$ such that $1\leq s \leq r$ and
all divisors $h$ of $\frac {p^k -1}{(2,1,1)}$, where $k=(r,s)$ and $(2,1,1)$ is defined as 2, 1 or 1 according as $p>2$ and $\frac{r}{k}$ is even, $p>2$ and $\frac{r}{k}$ is odd, or $p=2$.

\item For $q$ odd or $q=4^m$,\  alternating groups $A_4$,
of order 12.

\item For $q \equiv \pm 1 ( 8)$, symmetric groups $S_4$, of order 24.

\item For $q \equiv \pm 1 ( 5)$ or $q=4^m$, alternating groups $A_5$, of order 60. For $q \equiv 0 ( 5)$, the groups $A_5$ are listed under family~10.

\item $\frac{q(q^2 -1)}{ p^w (p^{2w}-1)}$ groups ${\rm L}_2(p^w)$, for all divisors $w$ of
$r$. The number of conjugacy classes of these groups is $2$, $1$ or $1$ according as $p>2$
and $\frac {r}{w}$ is even, $p>2$ and $ \frac {r}{w}$ is odd, or $p=2$.

\item groups ${\rm PGL}_2(p^w)$,
for all $w$ such that $2w$ is a divisor of $r$.
\end{enumerate}
\end{theorem}

\subsection{Class III maps for ${\rm L}_2(q^3)$ with $q$ odd.}

Let $p$ be an odd prime and let $q=p^n$ with $n$ a positive integer.
Let $G < {\rm Aut}({\rm L}_2(q^3))\cong {\rm PGL}_2(q^3):C_{3n}$ with  $G \cong {\rm L}_2(q^3)$.
Since $G$ is defined over a field ${\rm GF}(q^3)$, there exist outer automorphisms of $G$, of order 3.
Each of these outer automorphisms centralizes a subfield subgroup isomorphic to ${\rm L}_2(q)$.
We want to find a triple of involutions $(\rho_0, \rho_1, \rho_2)\subset G$ and an outer automorphism $\alpha$ of order 3 such that
\begin{enumerate}
\item $\alpha$ cyclically permutes $\rho_0$, $\rho_2$ and $\rho_0\rho_2$ (and therefore normalizes $\langle \rho_0,\rho_2\rangle$);
\item $\alpha$ fixes $\rho_1$ (see Figure~\ref{classes});
\item $\rho_0$ and $\rho_2$ commute;
\item $\langle \rho_0,\rho_1,\rho_2\rangle = G$;
\item there is no element of ${\rm Aut}(L_2(q^3))$ that swaps $\rho_0$ and $\rho_2$ and fixes $\rho_1$ (that is, the map has no duality).
\end{enumerate}

Point (3) says that $\langle \rho_0,\rho_2\rangle$ is an elementary abelian group  $E_4$ of order 4. Theorem~\ref{sub} tells us that every group ${\rm L}_2(q^3)$ has either one or two conjugacy classes of subgroups isomorphic to $E_4$.
When there are two classes of $E_4$, the classification of the subgroup structure of ${\rm PGL}_2(q^3)$, available for instance in~\cite{COT2006} or~\cite{Moo1904}, shows that the classes are fused under the action of ${\rm PGL}_2(q^3)$. Hence, up to isomorphism, there is a unique choice of a subgroup $E_4$ in $G$.
So we can pick an arbitrary subgroup $E_4$ in $G$, meaning we choose $\rho_0$, $\rho_2$ (and $\rho_0\rho_2$). 

Once the subgroup $E_4$ is chosen, in order to have point (1), we look for an element $\alpha$ of order 3 in ${\rm Aut}(L_2(q^3))\setminus G$ such that $\alpha$ is permuting the three involutions of $E_4$. 
Since $\alpha$ is of order three and not in $G$, it cannot be in ${\rm PGL}(2,q^3)$ and hence it must be a field automorphism. From Theorem~\ref{pslsub}, we readily see that the normalizer $N_{{\rm PGL}_2(q^3))}(E_4) \cong S_4$.
Therefore $N_{{\rm Aut}({\rm L}_2(q^3))}(E_4) \cong S_4 : C_{3n}$ where $C_{3n}$ is a subgroup of ${\rm Aut}(L_2(q^3)) \cong {\rm PGL}_2(q^3):C_{3n}$ intersecting $G$ in the identity element only.
Thus an element $\alpha$ as above must exist.
%klara: reference or argument missing. 
We pick $\alpha$ in a subgroup $C_{3n}$ that normalizes $E_4$ but does not centralize it.

In order to have point (2), we need to find an involution $\rho_1\in G$ such that $\rho_1^\alpha = \rho_1$, or equivalently $\alpha^{\rho_1} = \alpha$. So $\rho_1\in C_G(\alpha)$.  
As $\alpha$ is a field automorphism of order 3, it fixes a subfield $GF(q)$ of $GF(q^3)$ and therefore the centralizer of $\alpha$ in $G$ is a subgroup $H\cong {\rm L}_2(q)$. So every involution of $H$ is a candidate for $\rho_1$.
%klara: missing reference or argument.

Now let us focus on point (4).
No involution $\rho_1$ of $H$ commutes with $\rho_0$ as if it did, $\rho_1^\alpha = \rho_1$ would also commute with $\rho_0^\alpha = \rho_2$ and therefore, $\langle \rho_0,\rho_2,\rho_1\rangle \cong E_8$ but $G$ does not have such subgroups by Theorem~\ref{sub}.
So every involution $\rho_1$ of $H$ gives us a triple $(\rho_0, \rho_1, \rho_2)$ of involutions, two of which commute, namely $\rho_0$ and $\rho_2$ and such that $\rho_1$ does not commute with $\rho_0$ nor with $\rho_2$.
The subgroup $\langle \rho_0,\rho_1,\rho_2\rangle$ contains an elementary abelian subgroup of order 4 and two distinct dihedral subgroups, namely $\langle \rho_0,\rho_1\rangle$ and $\langle \rho_1,\rho_2\rangle$ that are isomorphic (as $\langle \rho_0,\rho_1\rangle^\alpha = \langle \rho_1,\rho_2\rangle$).
Looking in the list of subgroups of ${\rm L}_2(q^3)$ of Theorem~\ref{sub} for such subgroups, we readily see that types (1), (2), (5) and (6) do not have elementary abelian subgroups of order 4. 
Suppose $\langle\rho_0,\rho_1,\rho_2\rangle$ is of type (3). Then it is a dihedral group and in that case, as $\rho_0$ and $\rho_2$ commute, one of them must be in the centre of the dihedral group. Then it commutes also with $\rho_1$. 
But then, $\rho_2 = \rho_0^\alpha$ commutes also with $\rho_1$ and the group $\langle\rho_0,\rho_1,\rho_2\rangle$ has to be elementary abelian. Hence it has to be isomorphic to $E_4$ and $\rho_1 = \rho_0\rho_2$. But then $\rho_1$ is not fixed by $\alpha$, a contradiction.
This last argument also excludes type (4) of Theorem~\ref{sub}.
The set of all involutions of a group $A_4$ generates a subgroup of order 4, excluding type (7) of Theorem~\ref{sub}.
Case (8) gives subgroups isomorphic to $S_4$. If $S_4$ has a map of Class III, then it must have three subgroups $D_{2n}$ for a certain $n$ (that can be either $3$ or $4$ in the case of $S_4$) that contain a same subgroup of order 2. Looking at the subgroup lattice of $S_4$ (see for instance~\cite[Page 23]{BDL96}), we readily see that this is not possible.
%A quick {\sc Magma}~\cite{Magma} computation shows that no map of $S_4$ admits trialities, which would be the case here because of the choice of $\alpha$. So this case cannot occur. The same argument holds for subgroups of 
Case (9) gives subgroups that are isomorphic to $A_5$. 
If $A_5$ has a map of Class III, then it must have three subgroups $D_{2n}$ for a certain $n$ (that can be either $3$ or $5$ in the case of $A_5$) that contain a same subgroup of order 2. Looking at the subgroup lattice of $A_5$ (see for instance~\cite[Page 26]{BDL96}), we readily see that this is not possible.
We are thus left with cases (10) and (11).
In other words, what could still happen is that ${\rm L}_2(q') \leq \langle \rho_0,\rho_2,\rho_1\rangle  < G$ with $q'$ a proper divisor of $q$.

Now, $\alpha$ centralizes $\rho_1$ and normalizes $E_4$.
We need a little lemma about normalizers of elementary abelian subgroups of order $4$ in ${\rm L}_2(q^3)$.
\begin{lemma}\label{nheqng}
Let $G= {\rm L}_2(q^3)$ and $K<G$ be isomorphic to ${\rm L}_2(q)$.
Let $E_4$ be an elementary abelian subgroup of order $4$ of $K$.
Then $N_K(E_4) = N_G(E_4)$. Moreover $N_G(E_4)\cong A_4$ or $S_4$. 
\end{lemma}
\begin{proof}
By Theorem~\ref{sub}, case (4), 
the number \sout{$n_G$} of elementary abelian subgroups of order 4 of $G$ is equal to 
$\frac{q^3(q^6 -1)}{12 (2,q^3-1)}$,
the number \sout{$n_H$} of elementary abelian subgroups of order 4 of $K$ is equal to 
$\frac{q(q^2-1)}{12 (2,q-1)}$, and if $q \equiv \pm 1 (
8$), then $q^3\equiv \pm 1 (
8)$ as well while if $q \equiv \pm 3 (
8$), then $q^3\equiv \pm 3 (
8)$. In other words, the number of conjugacy classes of subgroups $E_4$ in $K$ is the same as the number of conjugacy classes of subgroups in $G$ and in both cases, $|N_K(E_4)| = |N_G(E_4)|$, implying that $N_K(E_4) = N_G(E_4)$.
The fact that $N_K(E_4) \cong A_4$ or $S_4$ directly follows from the size of the conjugacy classes of subgroups isomorphic to $E_4$ given in case (4) of Theorem~\ref{sub}.
\end{proof}

Suppose first that $q$ is not a square (this makes groups of type (11) in Theorem~\ref{sub} non-existent).
By Lemma~\ref{nheqng}, if $K < G$ is a subgroup isomorphic to ${\rm L}_2(q)$, then $N_G(E_4) = N_K(E_4) \cong A_4$ or $S_4$. Let $x := |N_G(E_4)|$.
In both cases, the number of subgroups conjugate to $K$ in $G$ is $\frac{(q^3+1)q^3(q^3-1)}{(q+1)q(q-1)}$, the number of subgroups conjugate to $E_4$ in $G$ is $\frac{(q^3+1)q^3(q^3-1)}{2x}$, the number of subgroups conjugate to $E_4$ in $K$ is $\frac{(q+1)q(q-1)}{2x}$ and therefore, the number of subgroups conjugate to $K$ and containing $E_4$ is $\frac{(q^3+1)q^3(q^3-1)}{(q+1)q(q-1)}\cdot \frac{(q+1)q(q-1)}{2x} / \frac{(q^3+1)q^3(q^3-1)}{2x} = 1$.

Suppose next that $q$ is a square.
In that case, $G$ has subgroups of type (11) in Theorem~\ref{sub}. The subgroups isomorphic to ${\rm L}_2(q')$ with $q'$ not a square then have two conjugacy classes of subgroups $E_4$ that are fused in the subgroups ${\rm PGL}_2(q')$, hence the argument on the sizes of the normalizers in $G$ and in $H$ remains valid.
Therefore, there is a unique subgroup isomorphic to ${\rm L}_2(q)$ in ${\rm L}_2(q^3)$ that contains a given subgroup $E_4$.

%We just showed that every subgroup $E_4$ is contained in a unique subgroup isomorphic to ${\rm L}_2(q)$ of $G$.
Now, suppose that for every involution $\rho_1$ of $H$, the subgroup $\langle \rho_0,\rho_1,\rho_2\rangle$ is a proper subgroup of $G$.
Then $\langle \rho_0,\rho_1,\rho_2\rangle^\alpha = \langle \rho_0^\alpha,\rho_1^\alpha,\rho_2^\alpha\rangle= \langle \rho_2,\rho_1,\rho_0\rho_2\rangle=\langle \rho_0,\rho_1,\rho_2\rangle$, hence all these proper subgroups have to be contained in the unique subgroup isomorphic to ${\rm L}_2(q)$ which contains $E_4$.
So all the involutions of $H$ are in the same subgroup isomorphic to ${\rm L}_2(q)$ that contains $E_4$, meaning that $E_4$ is also in $H$. This is impossible as $\alpha$ permutes the involutions of $E_4$ and centralises $H$, a contradiction.
Hence, there must be at least one involution, say $\rho_1$ such that $\langle \rho_0,\rho_1,\rho_2\rangle = G$, proving that there exists a $\rho_1$ for which (4) holds.

%Continuing further the argument above, it may happen in the case that $q$ is a square that the group generated by $\rho_0$, $\rho_1$ and $\rho_2$ is a subgroup isomorphic to a ${\rm PGL}_2(q')$. But again, this would imply, if all groups generated by $\rho_0$, $\rho_1$ and $\rho_2$ with $\rho_1$ in $H$, that $E_4< H$, a contradiction.

%In summary, we showed that, when $q$ is odd, there must be at least one involution $\rho_1\in H$ such that $\langle \rho_0,\rho_1,\rho_2\rangle = G$. 
Now suppose we have a triple $(\rho_0,\rho_1,\rho_2)$ that satisfies (1) to (4) (such triples must exist by the discussion above).
It remains to show that (5) holds.
The triple $(\rho_0,\rho_1,\rho_2)$ gives a reflexive map $\mathcal M$ with a triality (given by $\alpha$). So this map is either of class III or IV, the latter happening if and only if $\mathcal M$ admits a duality. 
Suppose that $\mathcal M$ also has a duality. Then this duality, say $D$, would have to fix $\rho_1$ and swap $\rho_0$ with $\rho_2$ (see Figure~\ref{classes}).
Also, the Petrie operator $P$ would be such that $DP = \alpha$ or $\alpha^{-1}$. As $\alpha$ is in the outer automorphism group of $G$, one of $D$ and $P$ at least must also be in the outer automorphism group. 
But the outer automorphism group ${\rm Out}(G) \cong C_2\times C_n$ (where $q^3=p^n$) does not contain a subgroup isomorphic to $S_3$. So the map $\mathcal M$ cannot have dualities and trialities at the same time.
We thus have proven the following theorem.

\begin{theorem}\label{maintheo}
Let $q$ be a power of a prime. There exist involutions $\rho_0$, $\rho_1$ and $\rho_2$ in ${\rm L}_2(q^3)$ such that $\rho_0$ and $\rho_2$ commute, $\langle \rho_0, \rho_1, \rho_2 \rangle = {\rm L}_2(q^3)$ and the map generated by $\rho_0$, $\rho_1$ and $\rho_2$ is a map of class III.
\end{theorem}
The case when $q$ is even was proven by Jones and Poulton~\cite[Theorem 2.3]{JP2010} and the discussion above settles the case when $q$ is odd.
\begin{remark}\label{out}
The argument used to prove that the maps of Class III for ${\rm L}_2(q^3)$ do not have dualities is that the outer automorphism group contains a cyclic group of order three but no subgroup isomorphic to $S_3$. Other families of groups having this property exist and are therefore likely to also have Class III maps.
\end{remark}

\section{The maps of class III for ${\rm L}_2(q^3)$ - polynomial approach}

\begin{comment}
In~\cite[Theorem 3.2]{JP2010}, maps of class III are constructed for some groups ${\rm L}_2(q)$ with $q$ even.

\begin{theorem}~\cite{JP2010}
For each positive integer $e$ divisible by three, there is a regular map in class III with automorphism group ${\rm L}_2(2^e)$.
\end{theorem}
\end{comment}
%Computer experiments using {\sc Magma} gave us maps of class III for ${\rm L}_2(q)$ with $q=3^3, 5^3, 7^3, 9^3, 11^3, 13^3, 17^3, 19^3, 23^3$. This suggested to generalise the Theorem of Jones and Poulton as we did in the previous section.

The proof of Theorem~\ref{maintheo} is not constructive. We give in this section a series of equations that, if satisfied, give a map of class III for ${\rm L}_2(q^3)$ for a fixed prime power $q$.

Let $q = p^n$ for some prime $p$ and positive integer $n$.
Assume the field automorphism providing Wilson's triality is
\[\tau : \mathbb{F}_{q^3}\rightarrow \mathbb{F}_{q^3}: x \mapsto x^{q}.\]

The group ${\rm L}_2(q^3)$ is the quotient ${\rm SL}(2,q^3)/Z({\rm SL}(2,q^3))$.
Let $\theta$ be the application from ${\rm SL}(2,q^3)$ to ${\rm L}_2(q^3)$ that associates to a given 2 by 2 matrix of ${\rm SL}(2,q^3)$ the corresponding element of ${\rm L}_2(q^3)$ obtained by this quotient. In other words,

$\theta : {\rm SL}(2,q^3) \rightarrow {\rm L}_2(q^3) : g \mapsto \theta(g)$ such that if $g=\left(\begin{array}{cc}a&b\\c&d\end{array}\right)$, then $$\theta(g) : {\rm PG}(1,q^3)\rightarrow {\rm PG}(1,q^3) : x \mapsto \frac{ax+b}{cx+d}.$$

We want to find three matrices $R_0$, $R_1$ and $R_2 \in {\rm SL}(2,q^3)$ such that $\theta(R_0)$, $\theta(R_1)$ and $\theta(R_2)$ give us a reflexible map of class III for ${\rm L}_2(q^3)$.

As we have a unique conjugacy class of involutions in ${\rm L}_2(q^3)$ and all elements of order 4 in $C_{SL(2,q)}(\alpha)$ are conjugate, we can take $R_1 = \left(\begin{array}{cc}0&1\\-1&0\end{array}\right)$. 
Then, in order to have $R_0$ mapped on an involution by $\theta$, we need $tr(R_0) = 0$.
This gives $R_0 = \left(\begin{array}{cc}a&b\\c&-a\end{array}\right)$
and a first equation to ensure $R_0\in SL(2,q^3)$.
\begin{equation}\label{eq1}
    -a^2-bc = 1
\end{equation}

We want $\tau$ to map $R_0$ to $R_2$.
Hence $$R_2 = \left(\begin{array}{cc}a^{q}&b^{q}\\c^{q}&-a^{q}\end{array}\right),$$ and $$R_0R_2 = \left(\begin{array}{cc}a^{q+1}+bc^{q}&ab^{q}-ba^{q}\\ca^{q}-ac^{q}&cb^{q}+a^{q+1}\end{array}\right)$$.
We want $R_0$ and $R_2$ to commute, and so we impose 
\begin{equation}\label{eq3}
tr(R_0R_2) = a^{q+1}+bc^{q}+cb^{q}+a^{q+1} = 0
\end{equation}
Finally, we want $\tau$ to map $R_2$ onto $R_0R_2$.
Hence we get the following matrix equation.
\begin{equation}\label{eq4}
\left(\begin{array}{cc}a^{q^2}&b^{q^2}\\c^{q^2}&-a^{q^2}\end{array}\right) = \left(\begin{array}{cc}a^{q+1}+bc^{q}&ab^{q}-ba^{q}\\ca^{q}-ac^{q}&cb^{q}+a^{q+1}\end{array}\right)
\end{equation}
Observe that if Equation~\ref{eq4} is satisfied, then Equation~\ref{eq3} is also satisfied as the trace of Equation~\ref{eq4} is exactly Equation~\ref{eq3}.

In order for $\langle \theta(R_0),\theta(R_1),\theta(R_2)\rangle$ to generate ${\rm L}_2(q^3)$ we would like $R_0R_1$ to have order large enough. In particular, we would expect $R_0R_1=\left(\begin{array}{cc}-b&a\\a&c\end{array}\right)$ to be of order either $q^3-1$ or $q^3+1$. Therefore we get another equation where $\xi$ is a $(q^3\pm 1)$-root of unity in $GF(q^6)$.
\begin{equation}\label{eq2}
    -b+c = \xi + \xi^{-1}
\end{equation}
Observe that, in order to get elements of order $q^3+1$, one would need to find $\xi \in GF((q^3)^2)$ such that $-b+c = \xi+\xi^{-1}.
$

\begin{lemma}
If $R_0$, $R_1$ and $R_2$ satisfy Equations \ref{eq1}- \ref{eq2} above, then  $$\langle \theta(R_0), \theta(R_1), \theta(R_2) \rangle \cong L_2(q^3).$$
\end{lemma}
\begin{proof}
If Equation~\ref{eq2} is satisfied, then we know that, in ${\rm L}_2(q)$, the groups $$\langle \theta(R_0), \theta(R_1) \rangle \cong \langle \theta(R_2), \theta(R_1) \rangle \cong \langle \theta(R_0R_2), \theta(R_1) \rangle \cong D_{q^3\pm 1}$$ are maximal subgroups of ${\rm L}_2(q)$. Hence there are only two possibilities for $H := \langle \theta(R_0), \theta(R_1), \theta(R_2) \rangle$, namely $H\cong D_{q\pm 1}$ or $H\cong {\rm L}_2(q^3)$. If the former happens, we have that $\langle \theta(R_0), \theta(R_2) \rangle < \langle \theta(R_0), \theta(R_1) \rangle$, meaning that one of $\theta(R_0)$, $\theta(R_2)$ and $\theta(R_0R_2)$ must be the central involution of $D_{q\pm 1}$. But then one of $\langle \theta(R_0), \theta(R_1)\rangle$, $\langle \theta(R_2), \theta(R_1) \rangle$ or $\langle \theta(R_0R_2), \theta(R_1) \rangle$ must be a cyclic group, not a dihedral group, a contradiction.
\end{proof}

Equation~\ref{eq2} is there to ensure by the above Lemma that the group generated is indeed ${\rm L}_2(q^3)$ but experiments in {\sc Magma} shows that there should be other possible $n$th-roots of unity that can be used than $q^3-1$ and $q^3+1$.

%From the above equations, we can deduce the following.

%\begin{equation}\label{e1}
%    b^{q-1} = (a^{q^2}+a^{q+1})/(a^2+1)
%\end{equation}
%\begin{equation}\label{e2}
%    c = -(a^2+1)/b
%\end{equation}
%\begin{equation}\label{e3}
%    -b+c = \xi + \xi^{-1}
%\end{equation}
%where $\xi$ is a $(q\pm 1)$-root of unity in $GF(q^6)$.

%If $(a,b)$ satisfies Equation~\ref{e1}, then $(a,b i^{k(q^3-1)/(q-1)})$ also satisfies Equation~\ref{e1} for each $k=0, \ldots, q-2$.

\begin{comment}
\begin{center}
\begin{tabular}{||c|c|c|c|c|c||}
\hline
$q$&$q^3$&$\#$ Sols eq 5&Sols all eqs except trace&Numbers covered by $-b+c$&Numbers covered by $\xi + \xi^{-1}$\\
\hline
3&27&38&24&4&12\\
5&125&220&120&6&30 (48)\\
7&343&636&336&8&54 (138)\\
9&729&1352&720&10&144 (288)\\
11&1331&2530&1320&12&216 (432)\\
13&2197&4164&2184&14&360 (828)\\
17&4913&9616&4896&18&1124 (1872)\\
19&6859&13374&6840&20&1134 (2310)\\
\hline
\end{tabular}
\end{center}
From the computations it is clear that triples $(a,b,c)$ satisfying Equations~\ref{e1} and~\ref{e2} do not necessarily satisfy all equations except the constraint on the trace (Equation~\ref{e3}).

At least column 5 seems predictable (always equal to q+1?).

SO what I said in the last meeting (about -b+c and the $\xi$'s covering the whole $GF(q^3)$) is FALSE :-(

More computations: 
\end{comment}

Table~\ref{nsol} gives the number of solutions $(a,b,c)$ that give a map with triality and no duality of type $\{(q^3-1)/2,(q^3-1)/2\}_{(q^3-1)/2}$ for column 3 and a map of type $\{(q^3+1)/2,(q^3+1)/2\}_{(q^3+1)/2}$ for column 4.

\begin{table}
\begin{center}
\begin{tabular}{||c|c||c||c||}
\hline
$q$&$q^3$&
$\#$ Sols for $q^3-1$&
$\#$ Sols for $q^3+1$\\
\hline
3 & 27 & 0 & 24\\
4 & 64 & 12 & 12\\
5 & 125 & 48 & 36\\
7 & 343 & 96 & 96\\
8 & 512 & 162 & 108\\
9 & 729 & 192 & 192\\
11 & 1331 & 324 & 216\\
13 & 2197 & 432 & 576\\
16 & 4096 & 840 & 1680\\
17 & 4913 & 1440 & 624\\
19 & 6859 & 1020 & 1440\\
23 & 12167 & 2952 & 2160\\
\hline
\end{tabular}
\end{center}
\caption{Number of solutions}\label{nsol}
\end{table}

%Finding all solutions to equations~\ref{eq1} to~\ref{eq2} seems a very difficult task. Proving that these equations always admit at least one solution would already be enough to show that all groups ${\rm L}_2(q^2)$ are automorphism groups of at least one class III map. But even that seems very difficult.
%That said, asking a computational software like {\sc Magma}~\cite{magma} to find (all) solutions,to equations~\ref{eq1} to~\ref{eq4} is doable in a reasonable time for relatively small values of $q^3$.

\section{Wilson's triality and incidence geometry}\label{wilson}

Let $\mathcal M$ be a reflexible map of type $\{p,q\}_r$.
This map has faces that are $p$-gons, Petrie polygons that are $r$-gons and around a vertex, there are $q$ edges and $q$ faces forming a $q$-gon.
The map $\mathcal M$ has an automorphism group $G$ generated by three involutions $\rho_0$, $\rho_1$, $\rho_2$ such that 
$\rho_0\rho_1$ is of order $p$, $\rho_1\rho_2$ is of order $q$, $\rho_0\rho_2$ is of order 2 and $\rho_0\rho_2\rho_1$ is of order $r$ as explained in Section~\ref{definitions}. 

A natural coset geometry of rank four can be constructed from $\mathcal M$. 

\begin{construction}\label{cons}
Let $\mathcal M$ be a reflexible map with generators $\rho_0$, $\rho_1$ and $\rho_2$ stabilizing a chosen flag $F:=\{v,e,f\}$. 
Let $\Gamma(\mathcal M) := \Gamma(G;\{G_0,G_1,G_2,G_3\})$ be a rank four coset pregeometry where
\begin{itemize}
\item $G_0 := \langle \rho_0,\rho_1\rangle$ is the stabilizer of $f$.
\item $G_1 := \langle \rho_0,\rho_2\rangle$ is the stabilizer of $e$.
\item $G_2 := \langle \rho_1,\rho_2\rangle$ is the stabilizer of $v$.
\item $G_3 := \langle \rho_1,\rho_0\rho_2\rangle$ (it is the stabilizer of a Petrie polygon $p$ such that $v\in p$, $e\subset p$ and $f$ and $p$ have two edges in common).
\end{itemize}
\end{construction}
This coset pregeometry can also be constructed as a pregeometry where the elements are respectively the vertices, edges, faces and Petrie polygons of the map $\mathcal M$. 
Two elements of distinct type other than edges are incident if and only if there are two flags in the map containing both, or in other words, if there are exactly two edges incident to both of them. 
Therefore, a vertex is incident to a face or a Petrie polygon if and only if it is on the face or the Petrie polygon, while a face is incident to a Petrie polygon if and only if they have two edges in common. 
The incidence between edges and other elements is the one inherited from the map.

The following lemma is obvious. In order to prove it, we only need to prove that every flag of $\Gamma(\mathcal M)$ is contained in a chamber, which is straightforward.
\begin{lemma}
The coset pregeometry $\Gamma(\mathcal M)$ obtained from Construction~\ref{cons} is a geometry.
\end{lemma}
We now prove that the geometry $\Gamma(\mathcal M)$ is not flag-transitive, but almost.

\begin{lemma}
The group $G$ has two orbits on the chambers of $\Gamma(\mathcal M)$.
\end{lemma}
\begin{proof}
The coset geometry $\Gamma(\mathcal M)$ constructed in this way is transitive on the flags consisting of a face and a Petrie polygon as every coset geometry of rank two is flag-transitive.
There will be two orbits on the chambers, due to the fact that the residue of a face and a Petrie polygon has three vertices and two edges such that the vertices and the edges form a path of length 2 on the face and on the Petrie polygon.
As $\mathcal M$ is reflexible, there is an element of $G$ that exchanges the two vertices at the extremities of this path of length two while fixing the face and the Petrie polygon, but the vertex in the middle of the path cannot be exchanged with one of the vertices at the extremities. This finishes the proof.
\end{proof}

\begin{remark}
The geometry $\Gamma(\mathcal M)$ is clearly not firm.
Indeed, again, the residue of a Petrie polygon and an incident face consists of three vertices joined by two edges (a path). Hence some rank one residues will have cardinality 1.
\end{remark}

Wilson's classes of maps translate naturally in four classes of coset geometries constructed from maps. 
\begin{itemize}
\item If $\mathcal M$ is of Class I, 
then ${\rm Cor}(\Gamma(\mathcal M))/{\rm Aut}(\Gamma(\mathcal M)) \cong \langle \rangle$;
\item If $\mathcal M$ is of Class II,
then ${\rm Cor}(\Gamma(\mathcal M))/{\rm Aut}(\Gamma(\mathcal M)) \cong C_2$.
\item If $\mathcal M$ is of Class III,
then ${\rm Cor}(\Gamma(\mathcal M))/{\rm Aut}(\Gamma(\mathcal M)) \cong C_3$.
\item If $\mathcal M$ is of Class IV,
then ${\rm Cor}(\Gamma(\mathcal M))/{\rm Aut}(\Gamma(\mathcal M)) \cong S_3$.
\end{itemize}
As we mentioned earlier, Wilson thought in~\cite{Wilson1979} that Class III maps were very rare, and Jones and Poulton constructed infinitely many of them in~\cite{JP2010}.
Looking at the literature on incidence geometries, it seems to us that incidence geometries admitting trialities but no dualities are rare.
Construction~\ref{cons} permits us to construct infinitely many examples.

\begin{theorem}
Let $\mathcal M$ be a Class III map. Then $\Gamma(\mathcal M)$ has trialities but no dualities.
\end{theorem}
\begin{proof}
If $\mathcal M$ is of class III, $\mathcal M \cong DP(\mathcal M) \cong PD(\mathcal M)$ and $D(\mathcal M) = P(\mathcal M) = DPD(\mathcal M)$.
In that case, $\mathcal M$ is $\Sigma^+$-invariant (where $\Sigma^+$ is the subgroup of order 3 of $\Sigma$, see Section~\ref{definitions}) and by~\cite[Lemma 2.1]{JP2010}, the map is of type $\{n,n\}_n$ for some positive integer $n$.
Moreover, for $\mathcal M$ to be in Class III, the automorphism group $G:=\langle \rho_0,\rho_1,\rho_2\rangle$ of $\mathcal M$ should have an element $\alpha$ in its automorphism group ${\rm Aut}(G)$ that permutes $\rho_0$, $\rho_2$ and $\rho_0\rho_2$ and fixes $\rho_1$. Hence $\alpha$ must normalize $G$, must also centralize $\rho_1$ and normalize $\langle \rho_0,\rho_2\rangle$. In other words, $A_4 \cong \langle \alpha,\rho_0,\rho_2\rangle < Aut(G)$.
This automorphism of order three gives a triality for $\Gamma(\mathcal M)$.
as it is interchanging $G_0$, $G_2$ and $G_3$ of $\Gamma(\mathcal M)$ and fixing $G_1$.

Suppose $\Gamma(\mathcal M)$ has a duality. This duality would necessarily swap two sets among the vertex-set, set of faces and set of Petrie polygons. This would translate in $\mathcal M$ having a duality or a Petrie-duality, and therefore $\mathcal M$ would be of Class IV, a contradiction.
\end{proof}

%If $\mathcal M$ is a map of type $\{m,m\}_m$ and $n$ is the number of vertices of $\mathcal M$, then the incidence geometry will have $3n$ elements and $3\cdot n\cdot m$ chambers. All of its rank two truncations are isomorphic (thanks to the triality). Every rank one residue has three elements. Therefore this geometry is thick.
%Moreover the geometry is residually connected. But, as pointed out above, it is not flag-transitive. Its chambers split in two orbits under the action of the automorphism group, one of size $2\cdot n\cdot m$ and the other of size $n\cdot m$. 
%The orbit of size $2\cdot n\cdot m$ consists of the triples containing one vertex, one face and one Petrie polygon such that they are pairwise incident to two edges each, but the intersection of these three pairs of edges is a single edge. The orbit of size $n\cdot m$ consists of the triples containing one vertex, one face and one Petrie polygon such that they all three are incident to the same pair of edges.

Looking at some atlases of geometries that have been published in the past (see for instance~\cite{BDL96, Lee2008, atlasthin, thinsz8}), the only examples we were able to find of geometries admitting a triality but no duality were geometries 8, 86, 182 and 183 in~\cite[Table 1]{thinsz8}.
In the present paper, the geometries with triality coming from Construction~\ref{cons} are not flag-transitive while the references above concern flag-transitive geometries.
It would be interesting to try to construct infinitely many examples of flag-transitive incidence geometries that admit trialities but no dualities.
As suggested in Remark~\ref{out}, groups $G$ with ${\rm Out}(G)$ having a cyclic subgroup of order three but no subgroup $S_3$ seem to be a good starting point to do that.

%{\color{yellow} Should we say something about this construction of an incidence geometry being possible for any map/polyhedron? You should be able to do this for any polytope indeed. I am also interested in understanding when/why  this is not a doubling construction. Or rather, why we cannot make a new version of the doubling construction out of this. }

\end{document}